\documentclass[11pt,a4paper]{article}

\usepackage[a4paper,margin=1in]{geometry}
\usepackage[utf8]{inputenc}
\usepackage{subfiles}
\usepackage{mathtools}
\usepackage[thinc]{esdiff}
\usepackage{array}
\usepackage{gensymb}
\usepackage{amssymb}
\usepackage{amsmath, bm}
\usepackage{amsthm}
\usepackage{amsfonts}
\usepackage{braket}
\usepackage[super]{nth}
\usepackage{caption}
\usepackage{subcaption}
\allowdisplaybreaks
\usepackage{placeins}
\usepackage{float}
\usepackage{tikz}
\usepackage{breqn}
\usepackage{hyperref}
\usepackage{bm}
\usetikzlibrary{shapes.geometric, arrows}
\usetikzlibrary{trees}
\usepackage[toc,page]{appendix}
\usetikzlibrary{quantikz2}
\usetikzlibrary{calc}
\usepackage{tikz-cd}
\usepackage{yfonts}
\usepackage{paralist}
\usepackage{circuitikz}
\usetikzlibrary{shapes.geometric}
\usepackage{relsize}
\usetikzlibrary{decorations.markings}
\usepackage{algorithm}
\usepackage{algpseudocode}
\usepackage[utf8]{inputenc}
\usepackage{physics}
\usepackage{nicematrix}
\usepackage{rotating}
\usepackage{booktabs}
\usepackage{subcaption}   
\usepackage{siunitx}
\usepackage{comment}

\usepackage[style=numeric, backend=biber,maxbibnames=99]{biblatex} 
\newcommand{\R}{\mathbb R}

\theoremstyle{definition}
\newtheorem{theorem}{Theorem}[section]
\newtheorem{corollary}{Corollary}[theorem]
\newtheorem{lemma}[theorem]{Lemma}
\newtheorem{proposition}[theorem]{Proposition}

\newtheorem{example}{Example}[section]
\newtheorem{definition}{Definition}[section]

\title{Global optimization of low-rank polynomials}
\author{Lloren\c{c} Balada Gaggioli \thanks{
LAAS - CNRS, Université de Toulouse, France. Emails: llorenc.balada.gaggioli@fel.cvut.cz, henrion@laas.fr, korda@laas.fr}\textsuperscript{\;\;,}\thanks{
Faculty of Electrical Engineering, Czech Technical University in Prague, Czechia}  
\and Didier Henrion\footnotemark[1]\textsuperscript{\;\;,}\footnotemark[2]
\and Milan Korda\footnotemark[1]\textsuperscript{\;\;,}\footnotemark[2]}

\date{\today}

\usepackage{xcolor}
\definecolor{dkblue}{rgb}{0,0,0.8}

\begin{document}

\maketitle

\begin{abstract}
    This work considers polynomial optimization problems where the objective admits a low-rank canonical polyadic tensor decomposition. We introduce LRPOP (low-rank polynomial optimization), a new hierarchy of semidefinite programming relaxations for which the size of the semidefinite blocks is determined by the canonical polyadic rank rather than the number of variables. As a result, LRPOP can solve low-rank polynomial optimization problems that are far beyond the reach of existing sparse hierarchies. In particular, we solve problems with up to thousands of variables with total degree in the thousands. Numerical conditioning for problems of this size is improved by using the Bernstein basis. The LRPOP hierarchy converges from below to the global minimum of the polynomial under standard assumptions.
\end{abstract}

\section{Introduction}

Polynomial optimization problems (POPs), which involve minimizing a polynomial function subject to polynomial constraints, are fundamental in many fields. However, they are generally non-convex and difficult to solve globally. A powerful approach to solving these problems is the \emph{moment-sum-of-squares (SOS) hierarchy}, a.k.a. the Lasserre hierarchy, proposed originally in \cite{Lasserre2001}. This method recasts the non-convex POP into a sequence of convex semidefinite programming (SDP) relaxations of increasing size. 
The quality of the relaxation, and its computational cost, is governed by a parameter called the relaxation order. The primal SDP problems are relaxations of the moment formulation of the POP, whereas the dual SDP problems generate guaranteed bounds on the global minimum from polynomial SOS decompositions of increasing degrees. Under mild assumptions, the sequence of SOS bounds converges to the true global optimum of the original problem, and the minimizers can be extracted from the moments.
See e.g. the recent overviews \cite{Henrion2020, Nie2023, Theobald2024} and references therein. 

Despite its theoretical convergence, the primary drawback of the moment-SOS hierarchy is its scalability. The size of the largest semidefinite block in the SDP relaxations (corresponding to the moment matrix in the primal or the Gram matrix for the SOS decomposition in the dual) grows rapidly. This size scales polynomially in the relaxation order, but the exponent in this scaling is the number of variables of the POP, rendering the method computationally infeasible for problems with a moderate number of variables. To address this scalability challenge, significant research has focused on exploiting problem structure, most notably sparsity. One prominent approach is \emph{correlative sparsity}. This technique constructs a graph where vertices represent the problem variables, and an edge connects two vertices if their product appears in a common monomial in the problem data. By computing a chordal extension of this graph, one can find a clique tree decomposition satisfying the running intersection property (RIP). This concept was originally developed for general sparse SDPs, building on matrix completion theory \cite{Grone1984}, see the survey \cite{Vandenberghe2015}. It has been adapted e.g. for operator-splitting and first order optimization algorithms \cite{Zheng2020}.

When applied to the moment-SOS hierarchy, this correlative sparsity pattern allows for a significant reduction in complexity \cite{Lasserre2006, Waki2006}, provided the cliques are small. Instead of a single, large semidefinite block corresponding to all POP variables, the problem is decomposed. In the primal formulation, this involves introducing a separate measure for each clique in the graph decomposition. Linear constraints are then enforced to ensure that the marginal moments of these measures are consistent on the intersections of the cliques. This decomposition effectively replaces the single large SDP block with multiple, smaller SDP blocks, one for each clique. Chordality of the graph ensures that the running intersection property holds and that the Correlative Sparse moment-SOS hierarchy (CSSOS), converges \cite{Lasserre2006, Grimm2007}. This structured approach has been successfully implemented in software packages like SparsePOP \cite{Waki2008}. It has been extended and applied to various problems, such as minimizing rational functions in computer vision \cite{Bugarin2015} or structural engineering design \cite{Handa2025}. Further refinements have been explored, such as decomposed structured subsets \cite{Miller2022} and extensions to polynomial matrix inequalities \cite{Zheng2023}.

A more recent approach is \emph{term sparsity}, which is detailed in \cite{Magron2023}. In this framework, the graph vertices correspond to the monomials (terms) appearing in the problem data, rather than the variables. An adjacency graph is constructed based on the terms arising in the input polynomials. This method,  referred to as the TSSOS (Term Sparsity SOS), generates a new converging hierarchy of SDP relaxations \cite{Wang2021}. The key feature is the iterative construction of block-diagonal matrices by completing connected components of these term-based graphs. This hierarchy of hierarchies has been implemented in libraries such as TSSOS for Julia \cite{Magron2021} and it has seen various extensions. These include refinements to reduce block sizes using combinatorial optimization \cite{Shaydurova2025}, the development of minimal sparsity frameworks for specific applications like the AC-OPF problem \cite{LeFranc2024}, and hybrid approaches that combine both correlative and term sparsity (CS-TSSOS) to exploit all available problem structures \cite{Wang2022}.

In this paper we propose a new kind of sparsity in polynomial optimization which is motivated by tensor decomposition techniques: \emph{low-rank sparsity}.
A polynomial $p \in \R[\mathbf{x}]$ in the $n$ variables $\mathbf{x}=(x_1, \ldots, x_n)$ has rank $r$ if it can be written as
\[
f(\mathbf{x}) = \sum_{l=1}^r \prod_{i=1}^n f_{l,i}(x_i)
\]
where each $f_{l,i} \in \R[x_i]$ is a univariate polynomial in the scalar variable $x_i$.
We say that $f$ has low-rank if $r$ is significantly smaller than $n$.

This representation is a direct application of a fundamental concept from multilinear algebra. A polynomial can be viewed as a \emph{tensor} within a high-dimensional space formed by the tensor product of spaces of univariate polynomials (e.g., $V_1 \otimes V_2 \otimes \dots \otimes V_n$, where $V_i = \mathbb{R}[x_i]$). In this framework, a polynomial is rank-one or separable if it can be written as a single product of univariate functions, i.e. $\prod_{i=1}^n f_{l,i}(x_i)$. Consequently, the definition of a rank-$r$ polynomial is an exact statement that $f$ can be decomposed into a sum of $r$ separable (rank-one) tensors. The smallest such $r$ is known as the tensor rank of $f$. This decomposition is not unique and it is widely known in the literature as the Canonical Polyadic (CP) decomposition, also called CANDECOMP (Canonical Decomposition) and PARAFAC (Parallel Factors Analysis), see \cite{kolda2009tensor} for a comprehensive overview of tensor decompositions and their applications. The premise of low-rank polynomial optimization is therefore to exploit cases where the polynomial $f$, when viewed as a tensor, has a rank $r$ that is significantly smaller than its ambient dimensions. The CP decomposition has a wide range of applications, from neuroscience \cite{cp_neuroscience} to the multiparticle Schrödinger operator \cite{cp_schrodinger}, and in this work we apply it to efficient global optimization of polynomials.

Unlike matrices, the best rank-$r$ approximation to a given tensor can fail to exist (sequences of factors can diverge while the fit improves). Regularization, warm starts, or switching formats typically resolves this in practice \cite{desilva2008}.
Mature toolboxes make these models practical: \textsc{Tensorlab} (Matlab), the Matlab \textsc{Tensor Toolbox}, and \textsc{TensorLy} (Python) provide CP/Tucker/TT implementations with many options \cite{tensorlab,baderkolda_toolbox,tensorly}. 

The main contribution of this paper is to introduce LRPOP, a moment-SOS hierarchy for global minimization of low-rank polynomials.
The outline of the paper is as follows. In Section \ref{sec:momsos} we recall the standard (dense) moment-SOS hierarchy for POP and its sparse variants CSSOS and TSSOS. These are based on standard graph theory concepts that we recall. In Section \ref{sec:lrpop} we introduce the LRPOP hierarchy, prove its convergence, and show that it scales linearly with respect to the number of variables.
Numerical experiments and comparisons with existing sparse hierarchies are reported in Section \ref{sec:examples}. We show that we can solve routinely low-rank POPs with number of variables and degree scaling by the thousands, provided the polynomial exhibits good numerical conditioning. Extensions to other tensor factorizations and rational optimization are described in the concluding Section \ref{sec:conclusion}.

\section{Moment-SOS hierarchy and sparsity}\label{sec:momsos}

\noindent\textbf{Sums of squares.} 
Let $\mathbf{x}=(x_1,\dots,x_n)$ be a tuple of variables and $f(\mathbf{x})$ a polynomial in $\R[\mathbf{x}]$. We say a polynomial is a \emph{sum of squares (SOS)} if there exist polynomials $q_1(\mathbf{x}),\dots,q_L(\mathbf{x})$ such that
\begin{equation}
    f(\mathbf{x})=\sum_{i=1}^L q_i(\mathbf{x})^2.
\end{equation}
If we can write a polynomial in this form, we automatically have a certificate of non-negativity. If $f(\mathbf{x})$ is a polynomial of degree $2d$, then we can define the vector of monomials of degree at most $d$ as $\mathbf{z}_d(\mathbf x)$. Then, checking the existence of the SOS decomposition is equivalent to checking if there exists a positive semi-definite (PSD) matrix $Q$, called the Gram matrix, satisfying
\begin{equation}
    f(\mathbf{x})=\mathbf{z}_d^T(\mathbf x) Q \mathbf{z}_d(\mathbf x).
\end{equation}

\medskip
\noindent\textbf{Moments.}
Take a monomial basis $\{\mathbf{x}^{\bm{\alpha}}\}$ of $\R[\mathbf{x}]$ indexing a sequence $\mathbf{y}=  (y_{\bm{\alpha}})$, with $\bm \alpha \in \mathbb N^n$ a vector of powers. Now define the linear functional $L_{\mathbf{y}}:\R[\mathbf{x}]\rightarrow \R$ as
\begin{equation}
    f(\mathbf{x})=\sum_{\bm{\alpha}}f_{\bm{\alpha}}\mathbf{x}^{\bm{\alpha}} \mapsto L_{\mathbf{y}}(f(\mathbf{x}))=\sum_{\bm{\alpha}} f_{\bm{\alpha}}y_{\bm{\alpha}}.
\end{equation}

We define the moment matrix $M_{k}(\mathbf{y})$ as the matrix indexed by the monomials up to degree $k$, such that the matrix entries are 
\begin{equation}
    M_{k}(\mathbf{y})_{\bm{\beta \gamma}}=L_{\mathbf{y}}(\mathbf{x}^{\bm{\beta}}\mathbf{x}^{\bm{\gamma}})=\mathbf{y}_{\bm{\beta}+\bm{\gamma}}.
\end{equation}

Similarly, with $g = \sum_{\bm{\alpha}}g_{\bm{\alpha}}x^{\bm{\alpha}}$, we can define the localizing matrix $M_k(g\mathbf{y})$ as
\begin{equation}
    M_{k}(g\mathbf{y})_{\bm{\beta \gamma}}=L_{\mathbf{y}}(g\mathbf{x}^{\bm{\beta}}\mathbf{x}^{\bm{\gamma}})=\sum_{\bm{\alpha}}g_{\bm{\alpha}}\mathbf{y}_{\bm{\alpha}+\bm{\beta}+\bm{\gamma}}.
\end{equation}

\subsection{Dense hierarchy}
Let us consider the polynomial optimization problem (POP)
\begin{equation}\label{eq:pop}
\begin{split}
\inf_{\mathbf{x}\in\R^n} \quad &f(\mathbf{x})\\
 \text{s.t.} \quad &g_{j}(\mathbf{x})\geq0, \quad j=1,\dots, m.
\end{split}
\end{equation}

We can rewrite the problem as 
\begin{align}
    (P):\quad p^* &= \inf \, \{ f(\mathbf{x}) \,\,| \,\, \mathbf{x} \in K \},
\end{align}

for the basic semialgebraic set $K = \{\mathbf{x} \in \mathbb{R}^n \,\,| \,\,  g_j(\mathbf{x}) \geq 0 \ \}$. Equivalently, we can also write the dual of problem $(P)$ as
\begin{align}
    (D):\quad \lambda^* = \sup \{\lambda \,\,|\,\, f(\mathbf{x})-\lambda \geq 0, \, \forall \mathbf{x}\in K \ \}.
\end{align}

Letting $g_0=1$, the moment relaxation of the POP \eqref{eq:pop} is the primal SDP problem
\begin{equation}\label{eq:moment-sdp}
	\begin{split}
		(P_k): \quad p_k= \inf_{\textbf{y}}  \quad & L_y(f) \\
		\text{s.t.} \quad & L_y(1)=1\\
		& M_{k-d_j}(g_j y)\succeq 0, \quad j=0,\dots,m;
	\end{split}
\end{equation}
where $d_j=\lceil\frac{\deg(g_j)}{2}\rceil$ and  $d\geq\max(\lceil\frac{\deg(f)}{2}\rceil,d_1,\dots,d_m)$. Then, the dual SDP yields the SOS optimization problem
\begin{equation}\label{eq:sos-sdp}
\begin{split}
(D_k): \quad \lambda_k=\sup_{\lambda , Q_j} \quad & \lambda \\
\text{s.t.} \quad & f(\mathbf{x})-\lambda = \sum_{j=0}^m{\sigma_j(\mathbf{x})g_j(\mathbf{x})}\\
 & \sigma_j(\mathbf{x})=\mathbf{z}_{k-d_j}(\mathbf{x})^T Q_j \mathbf{z}_{k-d_j}(\mathbf{x}), \quad Q_j\succeq 0, \quad j=1,\dots,m.
\end{split}
\end{equation}
Note that the constraint can also be expressed as $f(\mathbf x)-\lambda \in \mathcal{Q}(\mathbf{g})$, where 
\[
\mathcal{Q}(\mathbf{g}) = \{\sum_{j=0}^m \sigma_j(\mathbf{x})g_j(\mathbf{x}) : \sigma_j \text{ SOS}\}
\]
is the quadratic module generated by the constraints $\mathbf g$.
Assume the Archimedean condition holds, i.e. there exists $R$ such that $R^2-\|\mathbf{x}\|^2 \in \mathcal Q(\mathbf g)$.
Then, Putinar’s Positivstellensatz implies that the relaxations converge to the global minimum monotonically \cite{Lasserre2001}:

\begin{equation}
    p_k \ \uparrow\ p^*, \quad \lambda_k \ \uparrow\ p^*.
\end{equation}

We call this the \emph{dense moment-SOS hierarchy} because it couples all the variables together, creating PSD matrices of size $\binom{n+k}{k}$. If the problem has sparsity, $f$ and $g_j$ only couple specific subsets of variables or monomials, we can design more efficient SDP relaxations, the sparse hierarchies.

\subsection{Sparse hierarchies}

Given the problem \eqref{eq:pop}, build the \emph{correlative sparsity graph} $G(V,E)$, where $V=\{1,\dots,n\}$, such that each vertex corresponds to a variable, and $(i,i')\in E$ if $x_i$ and $x_{i'}$ appear in the same monomial in $f$ or in the same constraint $g_j$. Let us define some concepts of graph theory before describing the sparse hierarchies.

    A graph $G$ is a \emph {chordal graph} if all the cycles of four or more vertices have a \emph{chord}, an edge connecting two vertices of the cycle but which is not part of it.
    A  \emph{clique} is a subset of vertices of a graph $G$ such that all pairs of vertices in the clique are connected by an edge. 

Let $\widetilde{G}$ be a \emph{chordal extension} of $G$, built by adding chords (edges) until there are no remaining cycles of four (or more) vertices without a chord. Then there exist cliques $\mathcal{I}=\{I_1,\dots,I_N\}$ that satisfy the \emph{running intersection property} (RIP): for each pair of cliques that share a specific variable, there is a path of cliques connecting them, all of which contain the variable. Equivalently, if we pick all the cliques that contain a certain variable, this forms a connected graph. If the cliques are maximal (cliques that are not subsets of any other clique), then the clique decomposition is a \emph{tree}, an undirected graph that is connected and without cycle. This is what we call a \emph{clique tree}. Throughout this work, whenever referring to cliques $I_1,\ldots,I_N$ of a chordal graph, we assume that the cliques are maximal and hence form a clique tree.

For each edge $(a,b)$ of the clique tree, we define the separator $S_{ab}=I_a \cap I_b$. Assign the constraints $g_j$ to some clique $I_a$, containing all of its variables, and index them with $J_a\subset\{1,\dots, m\}$, which assigns inequalities to clique $a$. And finally, group the monomials of $f$ by cliques of the clique tree such that $f(\mathbf{x})=\sum_{a=1}^N f_a(x_{I_a})$, for $f_a(x_{I_a})\in \R[x_{I_a}]$. Then the correlative sparse moment relaxation  \cite{Lasserre2006} reads
\begin{equation}\label{eq:csmom-sdp}
\begin{split}
(P^{\mathrm{cs}}_k):\quad 
p^{\mathrm{cs}}_k \;=\; \inf_{\{y^{(a)}\}_{a=1}^N} \quad & \sum_{a=1}^N \,L_{\mathbf{y^{(a)}}}( f_a) \\
\text{s.t.}\quad 
& M^{I_a}_k\!\big(y^{(a)}\big)\ \succeq\ 0, \quad a=1,\dots,N,\\
& M^{I_a}_{k-d_j}\!\big(g_j\,y^{(a)}\big)\ \succeq\ 0,\quad a=1,\dots,N,\ \ j\in J_a,\\
& y^{(a)}\big|_{S_{ab}} = y^{(b)}\big|_{S_{ab}},\quad a,b=1,\dots,N,\\
& y^{(a)}_0=1,\quad a=1,\dots,N,\\
\end{split}
\end{equation}
where $y^{(a)}\in\R^{\varepsilon}$ for $\varepsilon=\binom{|I_a|+2k}{2k}$. Note that in the moment SDP we have to add overlap equalities, to ensure that  moments corresponding to shared variables coincide in different cliques. The dual SOS problem is
\begin{equation}\label{eq:cssos-sdp}
\begin{split}
(D^{\mathrm{cs}}_k):\quad 
\lambda^{\mathrm{cs}}_k=\ \sup_{\lambda,\ \{Q_{a,0}\},\ \{Q_{a,j}\}} &\quad  \lambda \\
\text{s.t.}\quad 
& \sum_{a=1}^N f_a(x_{I_a})\ -\ \lambda\ -\ \sum_{a=1}^N\ \sum_{j\in J_a}\sigma_{a,j}(x_{I_a})\,g_j(x_{I_a})\ =\ \sum_{a=1}^N \sigma_{a,0}(x_{I_a}),\\
& \sigma_{a,0}(x_{I_a})\ =\ \mathbf{z}_{k,I_a}(x)^{\!\top}\, Q_{a,0}\, \mathbf{z}_{k,I_a}(x), \quad a=1,\dots,N,\\
& \sigma_{a,j}(x_{I_a})\ =\ \mathbf{z}_{k-d_j,I_a}(x)^{\!\top}\, Q_{a,j}\, \mathbf{z}_{k-d_j,I_a}(x), \quad a=1,\dots,N,\,\ j\in J_a,\\
& Q_{a,0}\succeq 0,\quad Q_{a,j} \succeq 0, \quad a=1,\dots,N,\,\ j\in J_a,
\end{split}
\end{equation}
where we let $\mathbf{z}_{k,I_a}(x)$ be the vector of monomials in variables $x_{I_a}$ up to degree $k$. 

Assume a sparse Archimedean condition holds: there exists $R>0$ and polynomials
$\{q_a(x_{I_a})\}_{a=1}^N$ in the clique variables such that
$
R-\|\mathbf x\|_2^2=\sum_{a=1}^N q_a(x_{I_a})
$
with each $q_a$ belonging to the quadratic module generated by $\{g_j:\, j\in J_a\}$ on $x_{I_a}$. 
Together with the chordal construction (so that the cliques $\{I_a\}$ admit a clique tree and the overlap equalities in \eqref{eq:csmom-sdp} are enforced), the \emph{correlative sparse SOS (CSSOS)} hierarchy bounds the sparse hierarchy from below $p^{cs}_k\leq p_k$, and is monotone and convergent:
\begin{equation}
    p^{\mathrm{cs}}_k \ \uparrow\ p^*, \quad \lambda^{\mathrm{cs}}_k \ \uparrow\ p^*,
\end{equation}
as $k\to\infty$ \cite{Lasserre2006,Waki2006}. 

Besides correlative sparsity, we can also have a sparse polynomial due to the reduced number of monomials appearing, the so called term sparsity \cite{Wang2021}. In this case, we build a term sparsity graph, take the chordal extension and consider the cliques in a similar manner. Term sparsity yields a convergent hierarchy under analogous conditions, called \emph{term sparsity SOS (TSSOS)}, although they follow from a different type of graph construction.

In essence, CSSOS groups by variables, while TSSOS groups by monomials. When a polynomial has a small number of terms, TSSOS often yields smaller blocks; when variable interactions are localized but dense within cliques, CSSOS is preferable. The hybrid CS-TSSOS combines both \cite{Wang2022}. There are also other sparse variants like ideal sparsity \cite{Korda2024}, where the structure of the constraints is exploited.

\section{LRPOP hierarchy}\label{sec:lrpop}

Any polynomial can be written as $f(\mathbf{x}) = \sum_{l=1}^r \prod_{i=1}^n f_{l,i}(x_i)$ for some rank $r$, and number of variables $n$. Let us consider the POP \eqref{eq:pop}. We can define a new set of variables, $\mathbf{t}\in\R^{nr}$, defined recursively for $l=1,\dots, r$, as $$t_{l,1}=f_{l,1}(x_1) \text{ for } i=1, \quad t_{l,i}=t_{l,i-1}f_{l,i}(x_i) \text{ for } i=2,\dots n.$$
The equality constraints can be rewritten as $h_{l,i}(\mathbf{t},x_i)=0$ with $h_{l,1}(\mathbf{t},x_i)=t_{l,1}-f_{l,1}(x_1)$, and $h_{l,i}(\mathbf{t},x_i)=t_{l,i}-t_{l,i-1}f_{l,i}(x_i)$ for $i=2,\dots n$. Then $f(\mathbf{x})=\sum_{l=1}^r t_{l,n}$, so now we can set the POP as
\begin{equation}\label{eq:lrpop-pop}
\begin{split}
\inf_{\mathbf{x}\in\R^n,\mathbf{t}\in\R^{rn}} \quad &\sum_{l=1}^r t_{l,n}\\
 \text{s.t.} \quad 
 &h_{l,i}(\mathbf{t},x_i) =0, \quad l=1,\dots, r; \,\, i=1,\dots,n\\
 & g_j(x_j)\geq0, \quad j=1,\dots,n.
\end{split}
\end{equation}

With these variables, we can build a correlative sparsity graph that will have a symmetric structure which can be exploited to get a clique decomposition of small maximal clique size. Let us illustrate this with an example for $r=2, n=5$. 

\begin{example}
Consider the objective function
\[
\begin{aligned}
    f(\mathbf{x}) &= 12x_3+24x_1x_3-6x_2x_3+4x_3x_4+2x_2x_4-18x_3x_5 \\
    & -12x_1x_2x_3+8x_1x_3x_4-2x_1x_2x_4-36x_1x_3x_5+4x_2x_3x_4+9x_2x_3x_5-2x_2x_4x_5-6x_3x_4x_5 \\
    & -10x_1x_2x_3x_4+18x_1x_2x_3x_5+2x_1x_2x_4x_5-12x_1x_3x_4x_5\\
    &+12x_1x_2x_3x_4x_5\\
    & = (1+2x_1)(-2+x_2)(-x_3)(3+x_4)(2-3x_5) + (-1+x_1)(2x_2)(1+3x_3)(-x_4)(1-x_5)
\end{aligned}
\]
\end{example}

The initial correlative sparsity graph is depicted in Figure~\ref{fig:X_n}.

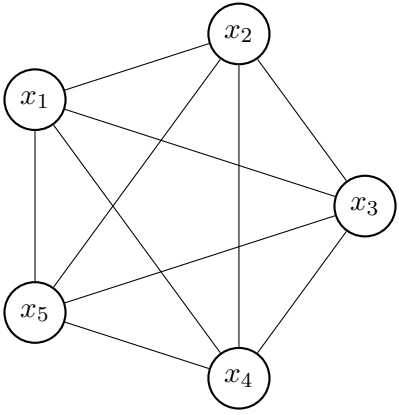
\begin{figure}[H]
\centering
\begin{tikzpicture}[
  varnode/.style={circle, draw, minimum size=8mm, thick},
  >=stealth
]
\def\radius{2.4cm}
\def\start{144}   
\def\step{-72}    

\foreach \k [evaluate=\k as \ang using {\start + (\k-1)*\step}] in {1,...,5}
  \node[varnode] (x\k) at (\ang:\radius) {$x_{\k}$};

\foreach \i/\j in {1/2,1/3,1/4,1/5, 2/3,2/4,2/5, 3/4,3/5, 4/5}
  \draw (x\i) -- (x\j);
\end{tikzpicture}

\caption{Correlative sparsity graph of the original problem.}
\label{fig:X_n}
\end{figure}

This is the complete graph $K_5$, so we cannot decompose it efficiently into small cliques. Similarly, the monomial graph will be quite complicated and very dense in edges, as the degree of the largest monomial is $\prod_{i=1}^n d_{l,i}$. Therefore both the CSSOS and TSSOS hierarchies do not bring about a significant reduction in problem size. 

\begin{definition}
    Let $G_{r,n}=(V,E)$ be the correlative sparsity graph of problem \eqref{eq:lrpop-pop}, where the set of vertices is given by the set of variables, and two vertices are connected by an edge if they appear in the same constraint. The graph $G_{r,n}$ depends only on the rank $r$, and the number of variables, $n$, of $f(\mathbf{x})$.
\end{definition}

\begin{figure}[H]
    \centering
\begin{tikzpicture}[
  >=stealth,
  vtx/.style={circle, draw, thick, minimum size=8mm, inner sep=1pt, fill=white},
  solidedge/.style={thick},
  dashededge/.style={thick, dashed, line cap=round},
  dottededge/.style={thick, dotted, line cap=round}
]

\foreach \i/\x in {1/0, 2/2.4, 3/4.8, 4/7.2, 5/9.6} {
  \node[vtx] (t1\i) at (\x,  3.0) {$t_{1,\i}$};
  \node[vtx] (x\i)  at (\x,  0.0) {$x_{\i}$};
  \node[vtx] (t2\i) at (\x, -3.0) {$t_{2,\i}$};
}

\foreach \i in {1,2,3,4,5}{
  \draw[solidedge] (t1\i)--(x\i)--(t2\i);
}
\draw[solidedge] (t11)--(t12)--(t13)--(t14)--(t15);
\draw[solidedge] (t21)--(t22)--(t23)--(t24)--(t25);

\foreach \i/\j in {1/2,2/3,3/4,4/5}{
  \draw[solidedge] (t1\i) -- (x\j);
  \draw[solidedge] (t2\i) -- (x\j);
}

\draw[dashededge] (t12) -- (t21); 
\draw[dashededge] (t13) -- (t22); 
\draw[dashededge] (t13) -- (t24); 

\draw[dashededge] (t11) .. controls +(0.9,0) and +(0.9,0) .. (t21); 
\draw[dashededge] (t12) .. controls +( 0.9,0) and +( 0.9,0) .. (t22); 
\draw[dashededge] (t13) .. controls +(0.9,0) and +(0.9,0) .. (t23); 
\draw[dashededge] (t14) .. controls +( 0.9,0) and +( 0.9,0) .. (t24); 

\end{tikzpicture}
    \caption{Correlative sparsity graph, $G_{2,5}$, with chords corresponding to a chordal extension shown in dashed lines.}
    \label{fig:X_rn}
\end{figure}
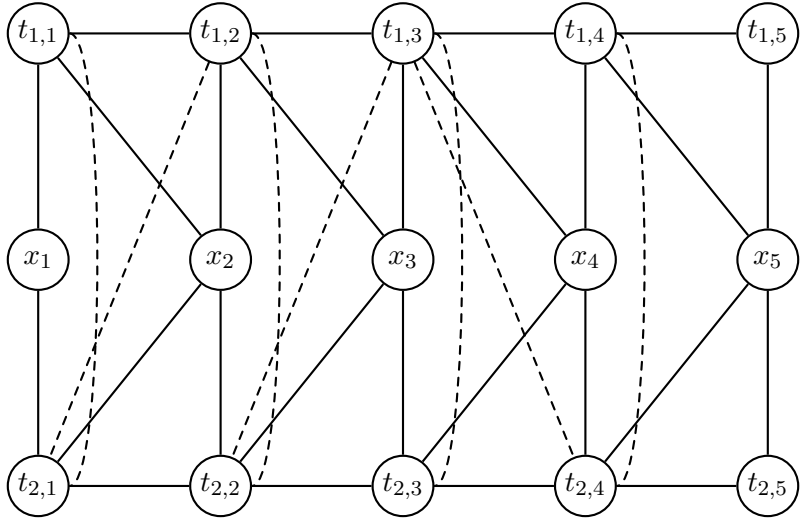

Now we take this new graph, $G_{r,n}$, we find a chordal extension as in Figure~\ref{fig:X_rn}, and then compute a clique decomposition to obtain the decomposition with the smallest maximal clique (depicted in Figure~\ref{fig:Tree}). For our example we have the largest clique of size 4, which is smaller than the maximal clique size with the other sparse methods. Note that this is a minimal example where one obtains a speedup. Far more significant speedup is obtained if the number of variables grows while the rank is kept fixed. This will be formalized in Theorem~\ref{TW}. 

\begin{figure}[H]
    \centering

\begin{tikzpicture}[
  scale=1.0, every node/.style={transform shape},
  bag/.style={
    draw, align=center,
    inner sep=3pt,
    minimum width=28mm,
    minimum height=9mm
  },
  >=Stealth
]
\node[bag] (C10) at (0,0) {$\{x_3,\, t_{12},\, t_{22},\, t_{23}\}$};

\node[bag] (C3)  [below left=8mm and 12mm of C10] {$\{x_2,\, t_{11},\, t_{12},\, t_{22}\}$};
\node[bag] (C2)  [below=9mm of C3]                 {$\{x_2,\, t_{11},\, t_{21},\, t_{22}\}$};
\node[bag] (C1)  [below=9mm of C2]                 {$\{x_1,\, t_{11},\, t_{21}\}$};

\node[bag] (C9)  [below right=8mm and 12mm of C10] {$\{x_3,\, t_{12},\, t_{13},\, t_{23}\}$};
\node[bag] (C8)  [below=9mm of C9]                  {$\{x_4,\, t_{13},\, t_{14},\, t_{23}\}$};
\node[bag] (C7)  [below=9mm of C8]                  {$\{x_4,\, t_{14},\, t_{23},\, t_{24}\}$};
\node[bag] (C6)  [below=9mm of C7]                  {$\{x_5,\, t_{14},\, t_{24}\}$};
\node[bag] (C5)  [below left=7mm and 0mm of C6]     {$\{x_5,\, t_{14},\, t_{15}\}$};
\node[bag] (C4)  [below right=7mm and 0mm of C6]    {$\{x_5,\, t_{24},\, t_{25}\}$};

\draw (C10) -- (C3);
\draw (C3)  -- (C2);
\draw (C2)  -- (C1);

\draw (C10) -- (C9);
\draw (C9)  -- (C8);
\draw (C8)  -- (C7);
\draw (C7)  -- (C6);
\draw (C6)  -- (C5);
\draw (C6)  -- (C4);
\end{tikzpicture}

    \caption{Clique tree decomposition of $G_{2,5}$.}
    \label{fig:Tree}
\end{figure}
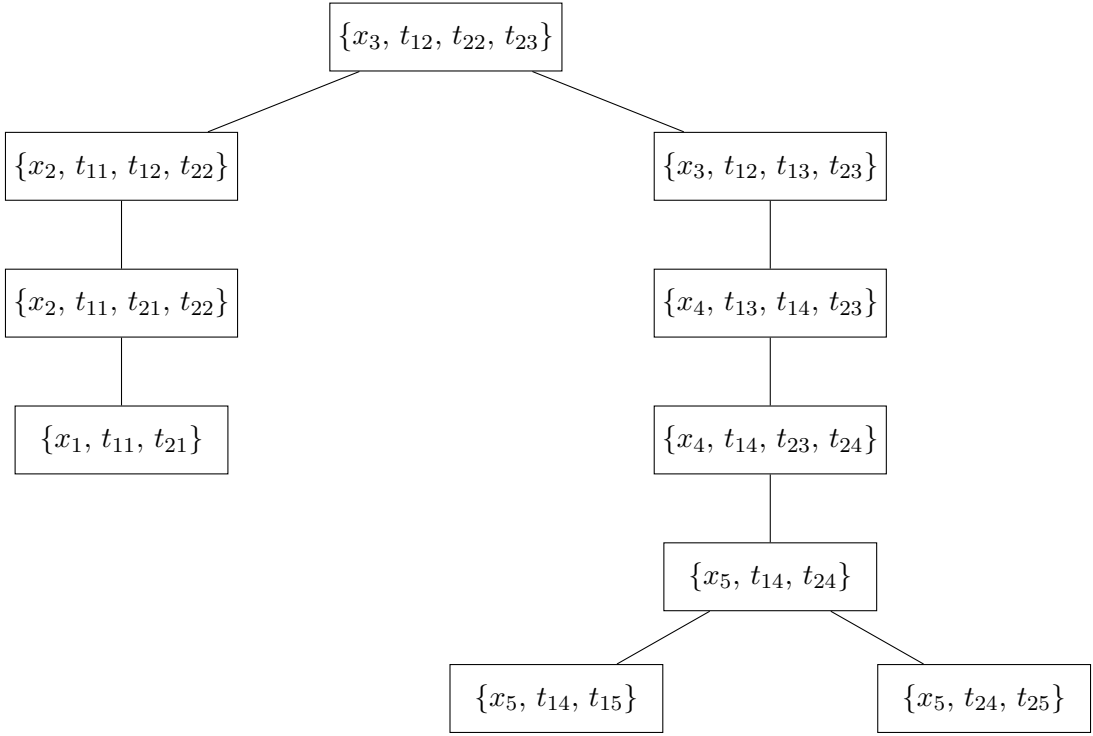

These cliques, named $I_1,\dots, I_N$, are used to form the moment and Gram matrices. The inequalities and equalities of the problem must now belong to a single clique. We define the inequalities at clique $I_a$ as $g(\mathbf{x},\mathbf{t};I_a)\geq 0$, and we let the index set $J_a\subset \{1,\ldots,n\}$ assign the inequalities for each clique. Similarly, we let the index set $H_a \subset\{(l,i) \mid l =1,\ldots,r, i = 1,\ldots,n\}$ assign the equalities for each clique. 

Let us define $q_a$ as the monomials in the variables $\{x_{I_a}\}$, such that we have, for a relaxation order $k$,  the relaxed moment SDP problem
\begin{equation}\label{eq:lrmom-sdp}
\begin{split}
(P^{\mathrm{LR}}_k):\quad 
p^{\mathrm{LR}}_k \;=\; \min_{\{y^{(a)}\}_{a=1}^N} \quad & \sum_{a=1}^N \,L_{\mathbf{y^{(a)}}}( f_a) \\
\text{s.t.}\quad 
& M^{I_a}_k\!\big(y^{(a)}\big)\ \succeq\ 0, \quad a=1,\dots,N,\\
& M^{I_a}_{k-d_j}\!\big(g_j\,y^{(a)}\big)\ \succeq\ 0,\quad a=1,\dots,N,\ \ j\in J_a,\\
& L_{\mathbf{y}^(a)}(q_ah_{l,i})=0, \quad (l,i)\in H_a,\, \deg(q_ah_{l,i})<2k,\\
& y^{(a)}\big|_{S_{ab}} = y^{(b)}\big|_{S_{ab}},\quad a,b=1,\dots,N,\\
& y^{(a)}_0=1,\quad a=1,\dots,N.\\
\end{split}
\end{equation}
The dual SOS SDP problem is
\begin{equation}\label{eq:lrpop-sdp}
\begin{split}
(D^{\mathrm{LR}}_k):\quad 
\lambda^{\mathrm{LR}}_k=\ &\sup_{\lambda,\ \{Q_{a,0}\},\ \{Q_{a,j}\},\ \{\tau_{a,l,i}\}} \quad  \lambda \\
\text{s.t.}\quad 
& \sum_{a=1}^N f_a(x_{I_a})\ -\ \lambda\ -\ \sum_{a=1}^N\ \sum_{j\in J_a}\sigma_{a,j}(x_{I_a})\,g_j(x_{I_a})\ \\
&-\sum_{a=1}^N\ \sum_{(l,i)\in H_a}\tau_{a,l,i}(x_{I_a})\,h_{l,i}(x_{I_a})\ =\ \sum_{a=1}^N \sigma_{a,0}(x_{I_a}),\\
& \sigma_{a,0}(x_{I_a})\ =\ \mathbf{z}_{k,I_a}(x)^{\!\top}\, Q_{a,0}\, \mathbf{z}_{k,I_a}(x), \quad \forall a\\
& \sigma_{a,j}(x_{I_a})\ =\ \mathbf{z}_{k-d_j,I_a}(x)^{\!\top}\, Q_{a,j}\, \mathbf{z}_{k-d_j,I_a}(x), \quad \forall a\,\ j\in J_a,\\
& Q_{a,0}\succeq 0,\quad Q_{a,j} \succeq 0, \quad \forall a\,\ j\in J_a,\\
& \tau_{a,l,i}(x_{I_a})\in \R[x_{I_a}], \,\, \deg(\tau_{a,l,i}h_{l,i})<2k, \quad \forall a,\forall l, \forall i, (l,i)\in H_a.
\end{split}
\end{equation}

We call this type of moment-SOS hierarchy LRPOP, for \emph{low-rank polynomial optimization}.

\subsection{Convergence}

To prove the convergence of this hierarchy, we follow the same steps as for the correlative-sparse hierarchy. First, we prove an auxiliary proposition related to the Archimedean property of the problem with the lifted variables.

\begin{proposition}
    If $K = \{\mathbf{x} \in \mathbb{R}^n \,\,| \,\,  g_j(\mathbf{x}) \geq 0 \ \}$ satisfies the Archimedean condition, then the lifted set $K' = \{\mathbf{x}\in \mathbb{R}^n, \mathbf{t} \in\R^{rn} \,\,| \,\,  g_j(\mathbf{x}) \geq 0, h_{l,i}(\mathbf{t},x_i)=0 \ \}$ also satisfies the Archimedean condition.
    \label{prop1}
\end{proposition}

\begin{proof}

Let $\R[\mathbf{x},\mathbf{t}]$ denote the ring of polynomials in the variables $\mathbf{x}$ and the lifted variables $\mathbf{t}$. We define the quadratic module generated by the inequality constraints $\mathbf{g}=\{g_1,\dots,g_m\}$ as
\[
\mathcal{Q}(\mathbf{g})=\left\{\sigma_0 + \sum_{j=1}\sigma_jg_j \,\, |\,\, \sigma_0,\dots,\sigma_j\in\Sigma[\mathbf{x,t}] \right\},
\]
where $\Sigma[\mathbf{x},\mathbf{t}]$ is the set of sum-of-squares polynomials. Similarly, we define the ideal generated by the equality constraints $\mathbf{h}=\{h_{l,i}\, |\, l=1,\dots,r;\,i=1,\dots,n\}$ as
\[
\mathcal{I}(\mathbf{h})=\left\{ \sum_{l,i} \phi_{l,i}h_{l,i}\, |\, \phi_{l,i} \in \R[\mathbf{x},\mathbf{t}] \right\}.
\]

Take the polynomial $f$ and assume $K = \{\mathbf{x} \in \mathbb{R}^n \,\,| \,\,  g_j(\mathbf{x}) \geq 0 \ \}$ satisfies the Archimedean condition, i.e. the quadratic module $\mathcal{Q}(\mathbf{g})$ generated by the $g_j$ contains the polynomial $R^2-\|\mathbf x\|^2$ for $R>0$ large enough.

Assume that $K$ satisfies the Archimedean condition. This implies that for $R>0$ large enough, the quadratic module generated by $\mathbf{g}$, restricted to $\mathbb{R}[\mathbf{x}]$, contains the polynomial $R^2 - \|\mathbf{x}\|^2$.
    
For the lifting variables $t_{l,j}$, let us define the bound $B_{l,j} := \sup_{x \in K} |\prod_{k}^j f_{l,k}(x_k)| + \varepsilon$ for some $\varepsilon > 0$. Since $K$ is compact (implied by the Archimedean property), this bound is finite. By Putinar's Positivstellensatz, we have:
\[
    B_{l,j}^2 - \left(\prod^j_k f_{l,k}(x_k)\right)^2 \in \mathcal{Q}(\mathbf{g}).
\]
Note that the equality constraints are $h_{l,j} = t_{l,j} - \prod^j_k f_{l,k}(x_k) = 0$. We can write:
\[
    B_{l,j}^2 - t_{l,j}^2 = \underbrace{B_{l,j}^2 - \left(\prod^j_k f_{l,k}(x_k)\right)^2}_{\in \mathcal{Q}(\mathbf{g})} + \underbrace{\left( \left(\prod^j_k f_{l,k}(x_k)\right)^2 - t_{l,j}^2 \right)}_{\in \mathcal{I}(\mathbf{h})}.
\]
The second term belongs to the ideal $\mathcal{I}(\mathbf{h})$ because $a^2 - b^2 = (a-b)(a+b)$, and $(a-b)$ is exactly our equality constraint $h_{l,j}$.
    
Summing over all indices $(l,j)$, and adding the bound for $\|\mathbf{x}\|^2$, we obtain:
\[
    R' - \|\mathbf{x}\|^2 - \|\mathbf{t}\|^2 \, \in \mathcal{Q}(\mathbf{g}) + \mathcal{I}(\mathbf{h})
\]
for some sufficiently large $R'$. Thus, the quadratic module associated with the lifted set $K'$ is Archimedean.
\end{proof}

\begin{corollary}
    Let $\mathcal{Q}_a$ denote the quadratic module generated by the inequality constraints $\{g_j\}_{j \in J_a}$ and the ideal of equality constraints $\{h_{l,i}\}_{(l,i) \in H_a}$ restricted to the variables of clique $I_a$. If the set $K$ satisfies the Archimedean condition, then for every clique $a=1,\dots,N$, the associated quadratic module $\mathcal{Q}_a$ also satisfies the Archimedean condition.
\end{corollary}
\begin{proof}
    Since the lifted set $K'$ is Archimedean by Proposition \ref{prop1}, take the variables associated with clique $I_a$ to be $x_{I_a}$, then there exists an $R>0$ such that
    \[
    R - \| \mathbf{x} \|^2 = R - \bigg(\| \mathbf{x}_{I_a} \|^2 + \| \mathbf{x}_{I_a^c} \|^2 \bigg) \, \in \mathcal{Q}(\mathbf{g})+\mathcal{I}(\mathbf{h}).
    \]

    Now, as $\| \mathbf{x}_{I_a^c} \|^2$ is a sum of squares polynomial, we can add it to this expression to get
    \[
    R - \bigg(\| \mathbf{x}_{I_a} \|^2 + \| \mathbf{x}_{I_a^c} \|^2 \bigg) +\| \mathbf{x}_{I_a^c} \|^2 = R-\| \mathbf{x}_{I_a} \|^2 \, \in \mathcal{Q}(\mathbf{g})+\mathcal{I}(\mathbf{h}).
    \]
    
    Therefore, the quadratic module generated by the constraints restricted to the variables in $I_a$ is Archimedean.
\end{proof}

\begin{theorem}
    Assume a sparse Archimedean condition holds: there exists $R$ and polynomials
$\{q_a(x_{I_a})\}_{a=1}^N$ in the clique variables such that
$
R^2-\|\mathbf x\|_2^2 = \sum_{a=1}^N q_a(x_{I_a})
$
with each $q_a$ belonging to the quadratic module generated by $\{g_j:\, j\in J_a\}$ on $x_{I_a}$. 
Together with the chordal construction (so that the cliques $\{I_a\}$ admit a clique tree and the overlap equalities in \eqref{eq:lrmom-sdp} are enforced), the low-rank hierarchy is monotone and convergent:
\begin{equation}
    p^{\mathrm{LR}}_k \ \uparrow\ p^*, \quad \lambda^{\mathrm{LR}}_k \ \uparrow\ p^*,
\end{equation}
as $k\to\infty$.
\end{theorem}

\begin{proof}
    Follows directly from the proposition that ensures the Archimedean condition holds for the lifted problem and the corollary that ensures the Archimedean condition holds for each clique, and the proofs in \cite{Lasserre2006,Waki2006}.
\end{proof}

\subsection{Maximal clique size}\label{Theorem}
The graph $G_{r,n}$ is fully determined by the rank and number of variables of $f(\mathbf{x})$. We are interested in the size of the largest positive semidefinite (PSD) block. For this let us introduce the notion of  tree decomposition, and treewidth, see; e.g. \cite{Diestel2025}.

A \emph{tree decomposition} of a graph $G=(V,E)$ is a tree $T$ where each node (vertex of the tree) is associated with a subset of vertices, a bag. Each bag of vertices $X_1,\dots,X_L$ satisfies the following properties
    \begin{itemize}
        \item Each vertex of the graph is contained in at least one bag and  $\bigcup_i X_i = V$.
        \item The bags containing vertex $v$ form a connected subtree of $T$.
        \item For every edge in the graph $(v,w)$, there is at least one bag containing both $v$ and $w$.
    \end{itemize}
The \emph{width} of a tree decomposition is the size of its largest bag minus one.
The \emph{treewidth} $\mathrm{tw}(G)$ of a graph $G$ is the minimum width among all possible tree decompositions of $G$.

As described for the sparse hierarchies, for each graph $G_{r,n}$ we can find a chordal graph containing it, that gives a clique tree decomposition. We are interested in a chordal extension that allows for a clique tree decomposition with its largest clique being  as small as possible. The size of the largest clique in this optimal clique tree decomposition is given by the treewidth as
\[\text{max clique size of } G_{r,n}=\mathrm{tw}(G_{r,n})+1.
\]
Let us now show that, for low-rank polynomials (with rank less than $n$), the treewidth of the $G_{r,n}$ graphs is fully determined by the rank $r$, and not the number of variables $n$, in sharp contrast with the dense and sparse hierarchies.

\begin{theorem}
    The variable graph $G_{r,n}$ has treewidth
    \[
    \mathrm{tw}(G_{r,n})=\min(n,r+1).
    \]
    In particular, when $r+1\leq n$, the maximal clique size in a chordal extension is at most $r+2$.
    \label{TW}
\end{theorem}

\begin{proof}
    For the upper bound we consider vertex elimination orders, which give us clique decompositions. If we find decompositions with maximal cliques of the size we want we will have an upper bound, as we cannot be sure it is the minimal decomposition. 

    A graph is chordal if and only if there exists a Perfect Elimination Ordering (PEO). A PEO is an ordering of vertices $v_1,v_2,\dots,v_m$ of the graph such that eliminating vertices in this order always leaves the remaining neighbors of the eliminated vertex forming a clique. Since a graph is chordal if and only if it has a PEO, we can use the definition of PEO to construct chordal graphs by considering a vertex eliminating ordering and adding edges at each step. This will ensure we have a PEO, therefore making the graph chordal.
    
   First let $r+1\leq n$, we show $\mathrm{tw}(G_{r,n})\leq r+1$. We want to build a PEO, so we consider the elimination ordering starting from the vertex $t_{1,n}$, we take all its neighbors and fully connect them. This forms a clique. Now we eliminate $t_{1,n}$ from the graph and keep the new edges. We then go to the second vertex of the elimination ordering, $t_{2,n}$, and do the same. We repeat this for each $t_{l,n}$ with $l=3,\dots,r$, and then we do it for $x_n$. When we reach $x_n$ we notice that it is connected to all $t_{l,n-1}$ for $l=1,\dots r$, so to $r$ many nodes. This creates a clique of size $r+1$. After connecting all of them we eliminate $x_n$ and move to $t_{1,n-1}$, and here we now have $r+1$ neighbors, making a clique of size $r+2$. We illustrate this for the first steps of the graph $G_{2,5}$ in Appendix~\ref{PEO}.
   
   The graph $G_{r,n}$ presents translational invariance across the columns so, by considering an adjacent pair of columns $i,i+1$, we know how the rest of the graph will work. The maximal clique for this vertex elimination ordering will be the one of $t_{1,n-1}$ with $r+2$ variables.

    Now let $n\leq r+1$, we show $\mathrm{tw}(G_{r,n})\leq n$. For this we now do the elimination row by row instead of column by column. Start from the vertex $t_{1,n}$ and connect all its neighbors between them, then eliminate the vertex. Keep going along the row $t_{1,i}$, and once we reach $t_{1,2}$ we note the vertex is connected to all $x_i$ for $i=2,\dots n$, so $n-1$ many of them, and to $t_{1,1}$. This makes a clique of size $n+1$. Following this elimination order and leaving the $x_i$ vertices for last we get a maximal clique of size $n+1$.

    With this, we finalized the proof of the upper bound. Let us move to the lower bound. First, let $n\leq r+1$,, we use the fact that minors of a graph cannot have larger treewidth than the full graph \cite{BODLAENDER2010259}. We can build a graph minor by contracting each of the $r$ rows into a single vertex. This creates a central set of vertices $x_i$ that are fully connected to each row vertex independently, but disconnected with each other. This has the form of a complete bipartite graph $K_{n,r}$, as we illustrate in Appendix~\ref{Bipartite}. The tree-width of a complete bipartite graph is $K_{n,r}=\min(n,r)=n$, therefore we have $\mathrm{tw}(G_{r,n})\geq n$. 

    Now let $r+1\leq n$, and consider any chordal extension, $G'$, of $G_{r,n}$. We show that $G'$ must contain a clique of size $r+2$ or bigger. Take the set of variables $S_i = \{t_{1,i},\dots, t_{r,i}\}$, then if we remove $S_i$ from $G_{r,n}$ we separate $x_i$ from $x_{i+1}$. Every minimal vertex separator has to induce a clique in any chordal extension \cite{berry:lirmm-00485851}, so $S_i$ will induce a clique in $G'$. Now consider the adjacent 'columns' $S_{i-1}$ and $S_i$, and the variable $x_i$, which is connected to all the elements in $S_{i-1}$ and $S_i$. We want to find the cliques follwowing from $G'$ so let us take a prefect elimination ordering of $G'$, and let $v$ be the first vertex of $S_{i-1}\cup S_i$ to appear in the ordering. Without loss of generality assume $v\in S_{i-1}$ (the same holds for $v\in S_i$). 
    
    Let $v=t_{l,i-1}$, then when $v$ is eliminated it will form a clique with its neighbors in $G'$ that appear later in the ordering. Since $S_{i-1}$ is a clique in $G'$ and $v$ is the first element in $S_{i-1}$ to be eliminated, it is connected to the remaining $r-1$ elements in the set. In $G_{r,n}$, $v$ is connected to $x_i$, so it will also be a neighbor in $S_i\cup S_{i-1}\cup \{x_i\}$. In the original graph, there is the edge $(t_{l,i-1},t_{l,i})$, so $t_{l,i}\in S_i$ is also a neighbor of $v$ still present in the graph. Therefore, the elimination of vertex $v$ requires a clique containing at least $S_{i-1}\cup \{x_i\}\cup \{t_{l,i}\}$, of size $r+2$. Since any chordal extension will have a clique of at least size $r+2$, the treewidth has to be bounded as $\mathrm{tw}(G_{r,n})\geq r+1$. We illustrate this in Appendix ~\ref{separator}.

    Combining the four segments of the proof, for the lower and upper bounds, we conclude $\mathrm{tw}(G_{r,n}) = \min(n,r+1)$.

\end{proof}

From this theorem we conclude that the graphs arising from polynomials of low rank (lower than the number of variables) give rise to a clique tree decomposition with maximal clique size $r+2$. This implies that the size of the PSD matrices is independent of the number of variables.

Note that the vertex elimination order we have used directly gives us a clique decomposition, so we can build the cliques from any of such graphs by following the same steps.

\subsection{Complexity}
We estimate the complexity of LRPOP by (i) counting the number and sizes of the PSD blocks, and (ii) counting linear equalities (overlap/separator constraints and lifting equalities). The LRPOP constraints are imposed clique-wise over a chordal extension of the low-rank graph $G_{r,n}$. Throughout this section we will assume $r+1\leq n$, such that $\mathrm{tw}(G_{r,n}) = r+1$. In particular, the maximal clique size is determined by the rank, not by $n$:
$\max_a |I_a|= \mathrm{tw}(G_{r,n})+1 = r+2$, and separators have size $|S_{ab}|\leq r+1$. This is the key structural reason why LRPOP scales well in $n$.

\medskip
\noindent\textbf{PSD blocks.}
Let $\{I_a\}_{a=1}^N$ be the cliques of a chordal extension of $G_{r,n}$. Each clique produces a Gram/moment block of dimension
\begin{equation}
s_a=\binom{|I_a|+k}{k}\ \leq\ \binom{r+2+k}{k},
\end{equation}
for relaxation order $k$. The total number of nodes of $G_{r,n}$ is $n(r+1)$, so the number of cliques must be lower or equal to it, therefore we have $N=\mathcal{O}(nr)$ many PSD blocks for the given sizes.

\medskip
\noindent\textbf{Overlap (separator) equalities.}
For each edge $(a,b)$ in the clique tree we enforce moment consistency on the separator $S_{ab}=I_a\cap I_b$. Matching the restricted moments up to degree $2k$ is equivalent to matching the shared moments between both moment matrices, which is equivalent to a restricted moment matrix of size
\begin{equation}
s_h=\binom{|S_{ab}|+k}{k}\ \leq\ \binom{r+1+k}{k}.
\end{equation}

This introduces $\tfrac{1}{2}s_h(s_h+1)$ scalar equalities per separator. Since a clique tree has $N-1=\mathcal{O}(nr)$ edges (one edge per node, except the root node), we obtain the count
\[
\mathcal{O}\Big(nr\,\binom{r+1+k}{k}^{2}\Big).
\]
At the minimal $k=\lceil\frac{d+1}{2}\rceil$, $\mathcal{O}(\binom{r+1+k}{k})=\mathcal{O}(r^k)$, so the initial relaxation order scales like $\mathcal{O}(n\,r^{d+2})$.

\medskip
\noindent\textbf{Lifting equalities.}
The lifted problem contains $rn$ polynomials $h_{l,i}=0$. In the moment side we impose
$L_{y^{(a)}}(q\,h_{l,i})=0$ for all monomials $q$ supported on a specific clique $I_a$ with $\deg(q\,h_{l,i})\leq 2k$. If $\deg f_{l,i}=d$, then $\deg h_{l,i}=d+1$, so the number of monomials per equality is
\begin{equation}
t_h\ =\ \binom{|I_a|+(2k-(d+1))}{\,2k-(d+1)\,}\ \le\ \binom{r+1+2k-d}{\,2k-d-1\,}.
\end{equation}
With $nr$ equalities $h_{l,i}$, this yields a total of
\[
\mathcal{O}\big(nr\,t_h\big)
\ =\ \mathcal{O}\Big(nr\,\binom{r+1+2k-d}{\,2k-d-1\,}\Big).
\]
At the minimal $k$ above, $t_h=\Theta(r^2)$, so the initial relaxation order scales like $\mathcal{O}(n\,r^{3})$. Therefore, in total we have $\mathcal{O}\Big(nr^{d+2}\Big)$ equalities for the initial relaxation order.

The cost of the LRPOP hierarchy scales only linearly with the number of variables, while the rank is the main driver of the cost. For the dense and sparse hierarchies, the number of variables (or monomials) of the cliques drive the cost.

\section{Numerical examples}
\label{sec:examples}

Let us consider examples for varying parameters. We want to study both how LRPOP compares to the dense hierarchy as we scale the parameters of the polynomials as well as how far we can go in the number of variables while keeping the rank low. For both studies we use the box constraints $x_i^2\leq 1$.  The numerical experiments were performed using the \textit{Julia} \cite{Julia-2017} packages \textit{LRPOP} \cite{lrpop} and \textit{TSSOS} \cite{Magron2021}. The SDP solver \textit{Mosek} \cite{Mosek} and \textit{CliqueTrees} \cite{cliquetrees2025samuelson} package were also a fundamental part of the code. The computer was a laptop with a CPU AMD Ryzen 7 PRO 7840U w/ Radeon 780M, 8 cores, base speed 3.30 GHz, and 32.0 GB of RAM.

\subsection{Comparison to dense hierarchy}

Starting with rank-1 polynomials, $f(\mathbf{x})=\prod_{i=1}^n f_i(x_i)$, for $f_i$ a univariate polynomial of degree~$d$. We pick the coefficients of the polynomials randomly: the coefficient for the $k$-th degree term is drawn from a normal distribution $\mathcal{N}(0, 0.7^{2k})$, and each $f_i$ is subsequently normalized such that $\sum_{k=0}^d |c_{i,k}| = 1$. The decay in the larger-degree coefficients will induce better conditioning for the numerical study. We fix a degree $d$, then vary $n$ and the relaxation order of the hierarchy, and we compare this to the dense hierarchy by using the TSSOS \cite{Magron2021} Julia package. The relaxation order of the low-rank method, $k_{\mathrm{LR}}$, will go from $\lceil\frac{d+1}{2}\rceil$, and we run TSSOS with the minimal relaxation order, $k_{\mathrm{dense}}=\lceil\frac{nd}{2}\rceil$. For the rank-1 example we fix the degree of the polynomials to $d=3$, with varying $n\geq 2$. If the computation takes more than 300 seconds we stop it and report a blank space in the tables of results.

\begin{table}[H]
\centering
\setlength{\tabcolsep}{5pt}
\renewcommand{\arraystretch}{1.1}
\caption{Optimal values for rank $r=1$, single polynomial degree $d=3$, and relaxation orders $k_{\mathrm{LR}}$, for LRPOP, and $k_{\mathrm{dense}}=\lceil\frac{nd}{2}\rceil$, for the dense hierarchy with TSSOS.}
\label{tab:opt_r1d3}
\begin{tabular}{
    l
    @{\hskip 10pt}
    *{5}{r}
}
\toprule
\#\,variables $n$ & {2} & {3} & {4} & {5} & {6} \\
\midrule
$k_{\mathrm{LR}}=2$ &  -0.710558 & -3.692186 & -3.238414 & -2.445621 & -7.363192\\
$k_{\mathrm{LR}}=3$ & -0.710559 & -0.305038 & -0.178236 & -0.030005 & -0.020383 \\

\midrule
Dense & -0.710559 & -0.305039 & -0.178236 &  &  \\
\bottomrule
\end{tabular}
\end{table}

\begin{table}[H]
\centering
\setlength{\tabcolsep}{5pt}
\renewcommand{\arraystretch}{1.1}
\caption{Running times (seconds) for rank $r=1$, single polynomial degree $d=3$, and relaxation orders $k_{\mathrm{LR}}$, for LRPOP, and $k_{\mathrm{dense}}=\lceil\frac{nd}{2}\rceil$, for the dense hierarchy with TSSOS.}
\label{tab:times_r1d3}
\begin{tabular}{
    l
    @{\hskip 10pt}
    *{5}{r}
}
\toprule
\#\,variables $n$ & {2} & {3} & {4} & {5} & {6} \\
\midrule
$k_{\mathrm{LR}}=2$ &  0.016  & 0.069  & 0.061  & 0.086  & 0.190 \\
$k_{\mathrm{LR}}=3$ & 0.052  & 0.070  & 0.082  & 0.141  & 0.106  \\

\midrule
Dense & 0.011  & 0.373  & 10.33  &  &  \\
\bottomrule
\end{tabular}
\end{table}

In Table \ref{tab:opt_r1d3} we see that we can reach the dense approximate value with a much lower relaxation order, and therefore also much faster. The SDP for the dense hierarchy does not solve it in less than 5 minutes for $d=3, n=5$, which we try to compute with the minimal relaxation order $k_{\mathrm{dense}}=8$, while we get what seems the optimal value for the LRPOP method with a relaxation order $k_{LR}=3$ within very short time, as we see in Table ~\ref{tab:times_r1d3}. The main advantage of the low-rank method is that we can get lower bounds with very little computational cost and for very large problems.

Now let us look at a polynomial with rank $r=2$ and degree $d=2$, and varying $n$. We show the results for the optimum values in Table~\ref{tab:opt_r2d2} and the running times in Table~\ref{tab:times_r2d2}.

\begin{table}[H]
\centering
\setlength{\tabcolsep}{5pt}
\renewcommand{\arraystretch}{1.1}
\caption{Optimal values for rank $r=2$, single polynomial degree $d=2$, and relaxation orders $k_{\mathrm{LR}}$, for LRPOP, and $k_{\mathrm{dense}}=\lceil\frac{nd}{2}\rceil$, for the dense hierarchy with TSSOS.}
\label{tab:opt_r2d2}
\begin{tabular}{
    l
    @{\hskip 10pt}
    *{5}{r}
}
\toprule
 \#\,variables $n$ & {2} & {3} & {4} & {5} & {6} \\
\midrule
$k_{\mathrm{LR}}=2$ & -1.280796 & -0.567000 & -1.083653 & -0.624894 & -0.610669\\
$k_{\mathrm{LR}}=3$ & -1.280796 & -0.482480 & -1.016807 & -0.234330 & -0.512391\\

\midrule
Dense & -1.280796 & -0.482480 & -1.016815 & -0.213803 &  \\
\bottomrule
\end{tabular}
\end{table}

\begin{table}[H]
\centering
\setlength{\tabcolsep}{5pt}
\renewcommand{\arraystretch}{1.1}
\caption{Running times (seconds) for rank $r=2$, single polynomial degree $d=2$, and relaxation orders $k_{\mathrm{LR}}$, for LRPOP, and $k_{\mathrm{dense}}=\lceil\frac{nd}{2}\rceil$, for the dense hierarchy with TSSOS.}
\label{tab:times_r2d2}
\begin{tabular}{
    l
    @{\hskip 10pt}
    *{5}{r}
}
\toprule
 \#\,variables $n$ & {2} & {3} & {4} & {5} & {6} \\
\midrule
$k_{\mathrm{LR}}=2$ & 0.027  & 0.571  & 0.075  & 0.108  & 0.133  \\
$k_{\mathrm{LR}}=3$ & 0.180 & 0.502  & 0.621  & 0.865 & 1.098 \\

\midrule
Dense & 0.008  & 0.065  & 0.328  & 10.36  &  \\
\bottomrule
\end{tabular}
\end{table}

We further extend this to higher rank, $r=4$, in Table~\ref{tab:opt_r4d3} and see that we still converge for low relaxation orders. However, the complexity is reflected in the running times to obtain the optimum, which are depicted in Table~\ref{tab:times_r4d3}.

\begin{table}[H]
\centering
\setlength{\tabcolsep}{5pt}
\renewcommand{\arraystretch}{1.1}
\caption{Optimal values for rank $r=4$, single polynomial degree $d=2$, and relaxation orders $k_{\mathrm{LR}}$, for LRPOP, and $k_{\mathrm{dense}}=\lceil\frac{nd}{2}\rceil$, for the dense hierarchy with TSSOS.}
\label{tab:opt_r4d3}
\begin{tabular}{
    l
    @{\hskip 10pt}
    *{5}{r}
}
\toprule
 \#\,variables $n$ & {2} & {3} & {4} & {5} & {6} \\
\midrule
$k_{\mathrm{LR}}=2$ & -1.321912 & -1.070262 & -1.725488 & -2.777503 & -1.597545 \\
$k_{\mathrm{LR}}=3$ & -1.321909 & -0.867154 & -0.918308 & -0.372227 & -0.512062 \\

\midrule
Dense & -1.321913 & -0.867158 & -0.918311 & -0.372247  &  \\
\bottomrule
\end{tabular}
\end{table}

\begin{table}[H]
\centering
\setlength{\tabcolsep}{5pt}
\renewcommand{\arraystretch}{1.1}
\caption{Running times (seconds) for rank $r=4$, single polynomial degree $d=2$, and relaxation orders $k_{\mathrm{LR}}$, for LRPOP, and $k_{\mathrm{dense}}=\lceil\frac{nd}{2}\rceil$, for the dense hierarchy with TSSOS.}
\label{tab:times_r4d3}
\begin{tabular}{
    l
    @{\hskip 10pt}
    *{5}{r}
}
\toprule
 \#\,variables $n$ & {2} & {3} & {4} & {5} & {6} \\
\midrule
$k_{\mathrm{LR}}=2$ & 0.029  & 0.596  & 1.093  & 1.489  & 2.480\\
$k_{\mathrm{LR}}=3$ & 0.233  & 1.305  & 11.29  & 32.84  & 80.39 \\

\midrule
Dense & 0.009  & 0.032  & 0.455  & 12.33  &  \\
\bottomrule
\end{tabular}
\end{table}

These examples all indicate that, for LRPOP, increasing the rank of the polynomial is the main driver of the complexity of the problem.

\subsection{Large-scale examples}

We can further test the scalability of LRPOP by seeing how far we can increase the number of variables. Increasing $n$ can generate ill-conditioned problems that are very challenging to solve without the use of higher precision arithmetic. To mitigate these effects, we use Bernstein polynomials \cite{FAROUKI2012379} to represent the POP. This allows for more stability in the numerics; see Appendix~\ref{bases} for details.

For $d\in \mathbb{N}$, the univariate Bernstein polynomials on $[0,1]$ are 
\begin{equation}
    B_{j,d}(s)=\binom{d}{j}s^j(1-s)^{d-j}, \quad j=0,\dots,d.
\end{equation}

They satisfy $\sum_{j=0}^d B_{j,d}(s)= 1$ for all $s$ and are nonnegative; therefore if we take any polynomial
\begin{equation}
    q(s)=\sum_{j=0}^d b_j B_{j,d}(s)
\end{equation}
then it satisfies $\min_j b_j \leq q(s)\leq \max_j b_j$ for all $s\in[0,1]$. For our numerical experiment we have the box constraints $x_i^2\leq1$ so we define the variables $s_i=(x_i+1)/2\in [0,1]$. We can write our univariate polynomials as $f_{l,i}(x_i)=\sum_{j=0}^d b_{l,i,j} B_{j,d}(s_i)$. For each rank $l$, we pick the coefficients 
\[
    b_{l,i,0}=1, \quad b_{l,i,j}\in [1+\tfrac{\delta}{n},1+\tfrac{2\delta}{n}],
\]

uniformly at random for $j=1,\dots, d$, and we let $\delta=1$. Picking $b_{l,i,0}=1$ ensures that the minimum of each univariate polynomial $f_{l,i}$ is $1$ and is attained at a common point $x_i = -1$. Therefore, the minimum of the full rank-$r$ polynomial is $\sum_{l=1}^r 1=r$. For example, for a polynomial of rank 2 we expect a global minimum $p^*=2$.

\begin{table}[H]
\centering
\setlength{\tabcolsep}{5pt}
\renewcommand{\arraystretch}{1.1}
\caption{Optimal values and computation times for $r=2$, $d=2$, $k_{\mathrm{LR}}=2$
and varying numbers of variables $n$ using LRPOP.}
\label{tab:bernstein_r2d2}
\begin{tabular}{
    l
    @{\hskip 10pt}
    *{5}{r}
}
\toprule
\#\,variables $n$ & $10$ & $50$ & $200$ & $500$ & $1000$ \\
\midrule
$k_{\mathrm{LR}}=2$ &  2.000016 &  2.000174 & 2.002931 & 2.015252 & 2.014466 \\
time (sec) & 1.141  & 5.980  & 26.44  & 49.33  & 112.6  \\
\bottomrule
\end{tabular}
\end{table}

Table~\ref{tab:bernstein_r2d2} summarizes the results. We note that even if theoretically we have proven that LRPOP converges to the global optimum via lower bounds, we get values higher than the optimum. This is due to the numerical conditioning of the problem for large number of variables. In the case $n=1000$ these numerical errors are less than $1\%$. The running time of LRPOP scales linearly with $n$, as expected from our complexity study.

\begin{table}[H]
\centering
\setlength{\tabcolsep}{5pt}
\renewcommand{\arraystretch}{1.1}
\caption{Optimal values and computation times for $r=3$, $d=2$, $k_{\mathrm{LR}}=2$ and varying numbers of variables $n$ using LRPOP.}
\label{tab:bernstein_r3d2}
\begin{tabular}{
    l
    @{\hskip 10pt}
    *{5}{r}
}
\toprule
\#\,variables $n$ & $10$ & $50$ & $200$ & $500$ & $1000$ \\
\midrule
$k_{\mathrm{LR}}=2$ &  3.000053 &  3.000692 & 3.019781 & 3.268124 & 3.284973 \\
time (sec) & 3.934 & 16.50 & 65.75  & 173.2  & 373.8  \\
\bottomrule
\end{tabular}
\end{table}

In this last example, summarized in Table~\ref{tab:bernstein_r3d2}, we expect a global minimum of $p^*=3$, and for $n=1000$ variables we see we have a $10\%$ error. This indicates the limits of LRPOP in terms of numerical conditioning.

These examples show how, for low-rank polynomials (e.g. $r=1,2,3$), LRPOP can compute the global minimum of polynomials with large number of variables and total degree. Our examples show this for 1000 variables and total degree of 2000, considering these are polynomials tailored specifically to have both low rank and good numerical conditioning.

\section{Conclusions and perspectives}\label{sec:conclusion}

In this paper, we introduced a novel approach to polynomial optimization by identifying and exploiting a low-rank tensor structure. This structure, distinct from correlative or term sparsity, provides a new avenue for tackling the scalability of the moment-SOS hierarchy. By tailoring the relaxations to this low-rank format, we open the door to solving a new class of large-scale problems where the underlying model, while dense with respect to existing notions of sparsity, possesses a low-rank structure of this kind. There are other rank-structured ways of writing polynomials, but not all may allow for moment-SOS hierarchies of this type. Exploring other low-rank structures by considering alternative \emph{tensor decompositions} is a natural way to continue this work and may widen the range of applications.

Another interesting direction for future research is to extend this low-rank approach from polynomial optimization to the broader class of \emph{rational optimization}. Minimizing rational functions is fundamental in many areas, including control theory, economics, and engineering design. While the moment-SOS hierarchy can be adapted for rational problems, it faces the same scalability issues.
How to address correlative sparsity for rational optimization was addressed in \cite{Bugarin2015}.
A promising perspective lies in bridging our low-rank optimization framework with modern data-driven approximation techniques that natively generate low-rank rational models. The {Loewner framework} has emerged as a powerful, data-driven method for system identification and function approximation \cite{Antoulas2005, Mayo2007}. Its extensions to the multivariate case are particularly relevant, as they construct rational approximations directly from function samples or data \cite{Gosea2022handbook,Antoulas2022sirev}. The core idea of this framework is to build a low-rank rational interpolant that, by construction, often takes a separable, low-rank tensor form analogous to the CP decomposition:
\[
f(x) \approx \sum_{l=1}^r \prod_{i=1}^n f_{l,i}(x_i)
\]
where the $f_{l,i}$ are now univariate rational functions.  A natural research question is, therefore, to develop a dedicated moment-SOS hierarchy that can efficiently solve optimization problems where the objective function or constraints are given in this low-rank rational form.

\section*{Acknowledgements}
We are grateful to Charles Poussot-Vassal for insightful information on tensor approximation and the Loewner framework.
The authors acknowledge the use of AI for assistance with brainstorming ideas, mathematical development, coding and drafting the manuscript. The final content, analysis and conclusions remain the sole responsibility of the authors. 

This work has been supported by European Union’s HORIZON–MSCA-2023-DN-JD programme under the Horizon Europe (HORIZON) Marie Skłodowska-Curie Actions, grant agreement 101120296 (TENORS). This work was also funded by the European Union under the project ROBOPROX (reg.~no.~CZ.02.01.01/00/22\_008/0004590).

\printbibliography{}

\appendix

\section{Vertex elimination ordering}
\label{PEO}

Here we illustrate the vertex elimination order defined in the proof of Theorem~\ref{TW} for $r=2, n=5$.

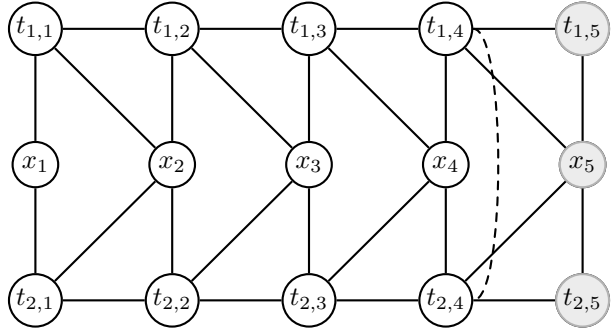
\tikzset{
  >=stealth,
  vtx/.style={circle, draw, thick, minimum size=6mm, inner sep=1pt, fill=white},
  solidedge/.style={thick},
  dashededge/.style={thick, dashed, line cap=round},
  ghost/.style={circle, draw=gray!60, thick, minimum size=6mm, inner sep=1pt, fill=gray!15},
  every node/.style={font=\small}
}

\newcommand{\BaseGraph}{

  \foreach \i/\x in {1/0, 2/1.8, 3/3.6, 4/5.4, 5/7.2} {
    \node[vtx] (t1\i) at (\x,  1.8) {$t_{1,\i}$};
    \node[vtx] (x\i)  at (\x,  0.0) {$x_{\i}$};
    \node[vtx] (t2\i) at (\x, -1.8) {$t_{2,\i}$};
  }

  \foreach \i in {1,2,3,4,5}{
    \draw[solidedge] (t1\i)--(x\i)--(t2\i);
  }
  \draw[solidedge] (t11)--(t12)--(t13)--(t14)--(t15);
  \draw[solidedge] (t21)--(t22)--(t23)--(t24)--(t25);
  \foreach \i/\j in {1/2,2/3,3/4,4/5}{
    \draw[solidedge] (t1\i) -- (x\j);
    \draw[solidedge] (t2\i) -- (x\j);
  }
}

\begin{figure}[H]
\centering
\begin{tikzpicture}
  \BaseGraph
  \node at (3.7,-3.0) {\textbf{Step 0:} original $G_{2,5}$};
\end{tikzpicture}
\end{figure}

\begin{figure}[H]
\centering
\begin{tikzpicture}
  \BaseGraph
  \node[ghost] at (t15) {$t_{1,5}$};
  \node at (3.7,-3.0) {\textbf{Step 1:} eliminate $t_{1,5}$};
\end{tikzpicture}
\end{figure}

\begin{figure}[H]
\centering
\begin{tikzpicture}
  \BaseGraph
  \node[ghost] at (t15) {$t_{1,5}$};
  \node[ghost] at (t25) {$t_{2,5}$};
  \node at (3.7,-3.0) {\textbf{Step 2:} eliminate $t_{2,5}$};
\end{tikzpicture}
\end{figure}

\begin{figure}[H]
\centering
\begin{tikzpicture}
  \BaseGraph
  \node[ghost] at (t15) {$t_{1,5}$};
  \node[ghost] at (t25) {$t_{2,5}$};
  \node[ghost] at (x5)  {$x_{5}$};

  \draw[dashededge]
    (t14) .. controls +(0.8,0) and +(0.8,0) .. (t24);
  \node at (3.7,-3.0) {\textbf{Step 3:} eliminate $x_{5}$, add edge $(t_{1,4},t_{2,4})$};
\end{tikzpicture}
\end{figure}

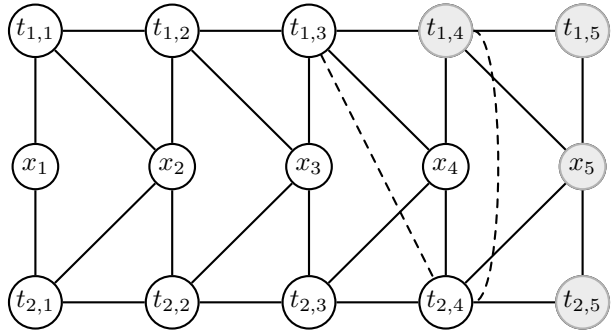
\begin{figure}[H]
\centering
\begin{tikzpicture}
  \BaseGraph
  \node[ghost] at (t15) {$t_{1,5}$};
  \node[ghost] at (t25) {$t_{2,5}$};
  \node[ghost] at (x5)  {$x_{5}$};

  \draw[dashededge]
    (t14) .. controls +(0.8,0) and +(0.8,0) .. (t24);

  \draw[dashededge] (t13) -- (t24);

  \node[ghost] at (t14) {$t_{1,4}$};
  \node at (3.7,-3.0) {\textbf{Step 4:} eliminate $t_{1,4}$, add edge $(t_{1,3},t_{2,4})$};
\end{tikzpicture}
\end{figure}

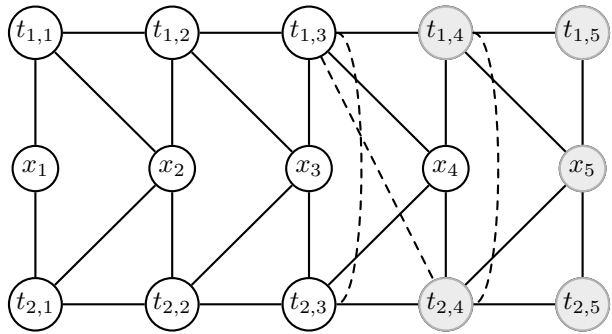
\begin{figure}[H]
\centering
\begin{tikzpicture}
  \BaseGraph
  \node[ghost] at (t15) {$t_{1,5}$};
  \node[ghost] at (t25) {$t_{2,5}$};
  \node[ghost] at (x5)  {$x_{5}$};
  \node[ghost] at (t14) {$t_{1,4}$};

  \draw[dashededge]
    (t14) .. controls +(0.8,0) and +(0.8,0) .. (t24);
  \draw[dashededge] (t13) -- (t24);

  \draw[dashededge]
    (t13) .. controls +(0.8,0) and +(0.8,0) .. (t23);

  \node[ghost] at (t24) {$t_{2,4}$};
  \node at (3.7,-3.0) {\textbf{Step 5:} eliminate $t_{2,4}$, add edge $(t_{1,3},t_{2,3})$};
\end{tikzpicture}
\end{figure}

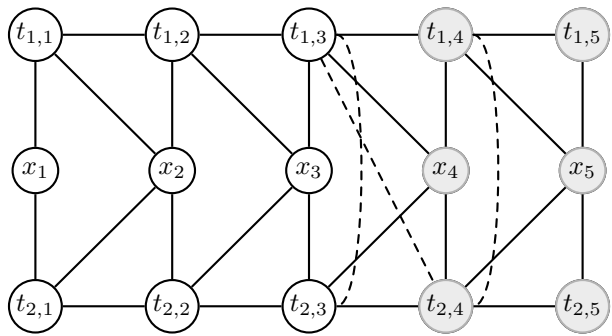
\begin{figure}[H]
\centering
\begin{tikzpicture}
  \BaseGraph
  \node[ghost] at (t15) {$t_{1,5}$};
  \node[ghost] at (t25) {$t_{2,5}$};
  \node[ghost] at (x5)  {$x_{5}$};
  \node[ghost] at (t14) {$t_{1,4}$};
  \node[ghost] at (t24) {$t_{2,4}$};

  \draw[dashededge]
    (t14) .. controls +(0.8,0) and +(0.8,0) .. (t24);
  \draw[dashededge] (t13) -- (t24);
  \draw[dashededge]
    (t13) .. controls +(0.8,0) and +(0.8,0) .. (t23);

  \node[ghost] at (x4) {$x_{4}$};
  \node at (3.7,-3.0) {\textbf{Step 6:} eliminate $x_{4}$};
\end{tikzpicture}
\end{figure}

We proceed in a similar manner until there are no vertices left in the graph.

\section{Bipartite graph minor}
\label{Bipartite}

We illustrate the bipartite graph minor defined in the proof of Theorem~\ref{TW} for $r=3,n=4$. We can visualize $G_{r,n}$ as a series of rows of $t_{l,i}$ for $i=1,\dots,n$ with a bottom row of $x_i$. Note that this is different from the pictures we did until now because when the rank is larger than two the following is a more natural arrangement of the vertices. 

\begin{figure}[H]
\centering
\begin{tikzpicture}[scale=0.85]

  \tikzstyle{vtx}=[circle, draw, thick, fill=white, inner sep=0pt, minimum size=20pt]
  \tikzstyle{solidedge}=[thick]
  \tikzstyle{rowbox}=[dashed, cyan, thick]

  \foreach \i/\x in {1/0, 2/2.5, 3/5, 4/7.5} {

    \node[vtx] (x\i) at (\x - 1.25, -2.5) {$x_{\i}$};

    \node[vtx] (t3\i) at (\x, -0.5) {$t_{3,\i}$};
    \node[vtx] (t2\i) at (\x,  1.5) {$t_{2,\i}$};
    \node[vtx] (t1\i) at (\x,  3.5) {$t_{1,\i}$};
  }

  \foreach \i in {1,2,3,4}{
    \draw[solidedge] (t3\i) -- (x\i);
    \draw[solidedge] (t2\i) -- (x\i);
    \draw[solidedge] (t1\i) -- (x\i);
  }

  \foreach \i/\j in {1/2, 2/3, 3/4}{
    \foreach \l in {1,2,3}{
      \draw[solidedge] (t\l\i) -- (t\l\j);
    }
  }

  \foreach \i/\j in {1/2, 2/3, 3/4}{
    \foreach \l in {1,2,3}{
      \draw[solidedge] (t\l\i) -- (x\j);
    }
  }

  \draw[rowbox] (-0.7, 2.6) rectangle (8.2, 4.4);
  \node[anchor=west, cyan] at (8.3, 3.5) {$\text{contract to } R_1$};

  \draw[rowbox] (-0.7, 0.6) rectangle (8.2, 2.4);
  \node[anchor=west, cyan] at (8.3, 1.5) {$\text{contract to } R_2$};

  \draw[rowbox] (-0.7, -1.4) rectangle (8.2, 0.4);
  \node[anchor=west, cyan] at (8.3, -0.5) {$\text{contract to } R_3$};

  \node at (4.5, -4.0) {$G_{3,4}$};

\end{tikzpicture}
\end{figure}

\begin{figure}[H]
\centering
\begin{tikzpicture}[scale=1.0]

  \tikzstyle{vtx}=[circle, draw, thick, fill=white, inner sep=0pt, minimum size=20pt]
  \tikzstyle{solidedge}=[thick]

  \node[vtx] (R1) at (0,  2.0) {$R_1$};
  \node[vtx] (R2) at (0,  0.0) {$R_2$};
  \node[vtx] (R3) at (0, -2.0) {$R_3$};

  \foreach \i/\y in {1/2.25, 2/0.75, 3/-0.75, 4/-2.25} {
    \node[vtx] (X\i) at (5, \y) {$x_{\i}$};
  }

  \foreach \r in {1,2,3} {
    \foreach \x in {1,2,3,4} {
        \draw[solidedge] (R\r) -- (X\x);
    }
  }

  \node at (2.5, -3.5) {$K_{3,4}\text{ minor}$};

\end{tikzpicture}
\end{figure}

\section{Clique from separator}\label{separator}

We examine the elimination of vertex $v=t_{2,i-1}$ (yellow). Solid black lines indicate edges present in the original graph $G_{r,n}$. Blue dashed lines indicate fill-in edges required because column $S_{i-1}$ is a minimal separator. When $v$ is eliminated, it must form a clique with all its current neighbors (blue background). The gray dashed lines represent the new edges from this elimination. Consequently, the set $S_{i-1} \cup \{x_i\} \cup \{t_{2,i}\}$ forms a clique of size at least $r+2$ (5 in this example).

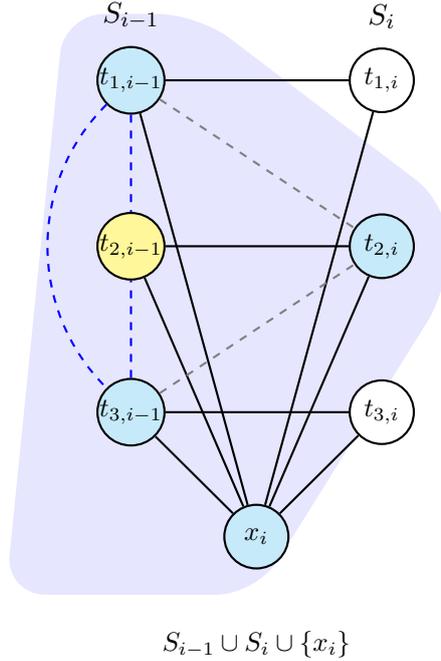
\begin{figure}[H]
\centering
\begin{tikzpicture}[scale=1.1]

  \tikzstyle{vtx}=[circle, draw, thick, fill=white, inner sep=0pt, minimum size=24pt]
  \tikzstyle{origedge}=[thick, black]
  \tikzstyle{fillinedge}=[thick, dashed, blue]
  \tikzstyle{impliededge}=[thick, dashed, gray]
  \tikzstyle{cliquebg}=[fill=blue!10, rounded corners=15pt]

  \draw[cliquebg, draw=none] 
    (-0.8, 3.8) -- (0.8, 3.8) -- (3.8, 1.5) -- (3.8, 0.5) -- (1.5, -3.2) -- (-1.5, -3.2) -- cycle;

  \node[vtx, fill=cyan!20] (t1_prev) at (0, 3) {$t_{1,i-1}$};
  \node[vtx, fill=cyan!20] (t3_prev) at (0, -1) {$t_{3,i-1}$};

  \node[vtx, fill=yellow!50] (v) at (0, 1) {$t_{2,i-1}$};

  \node[vtx] (t1_curr) at (3, 3) {$t_{1,i}$};
  \node[vtx, fill=cyan!20] (t2_curr) at (3, 1) {$t_{2,i}$};
  \node[vtx] (t3_curr) at (3, -1) {$t_{3,i}$};

  \node[vtx, fill=cyan!20] (xi) at (1.5, -2.5) {$x_i$};

  \draw[origedge] (t1_prev) -- (t1_curr);
  \draw[origedge] (v) -- (t2_curr);
  \draw[origedge] (t3_prev) -- (t3_curr);

  \draw[origedge] (t1_prev) -- (xi);
  \draw[origedge] (v) -- (xi);
  \draw[origedge] (t3_prev) -- (xi);

  \draw[origedge] (t1_curr) -- (xi);
  \draw[origedge] (t2_curr) -- (xi);
  \draw[origedge] (t3_curr) -- (xi);

  \draw[fillinedge] (t1_prev) -- (v);
  \draw[fillinedge] (t3_prev) -- (v);
  \draw[fillinedge] (t1_prev) to[bend right=45] (t3_prev);

  \draw[impliededge] (t1_prev) -- (t2_curr);
  \draw[impliededge] (t3_prev) -- (t2_curr);

  \node[anchor=south, font=\bfseries] at (0, 3.5) {$S_{i-1}$};
  \node[anchor=south, font=\bfseries] at (3, 3.5) {$S_i$};

  \node at (1.5, -3.8) {$S_{i-1} \cup S_i \cup \{x_i\}$};

\end{tikzpicture}
\label{fig:separator}
\end{figure}

\section{Bernstein polynomials and ill-conditioning}
\label{bases}

When minimizing polynomials that are sums of products of univariate factors (our low-rank
format) the choice of basis matters numerically. The monomial basis $\{1,x,\dots,x^d\}$, while algebraically simple, tends
to be numerically ill-conditioned for high-degree optimization problems on intervals. The Bernstein basis \cite{FAROUKI2012379} avoids this poor scaling
by its structure and total nonnegativity. We analyze the sensitivity of the mapping $\Phi$ from the coefficient space to the function values. A large Lipschitz constant for $\Phi$ implies that small numerical errors in the SDP variables correspond to large deviations in the objective, leading to solver failure.

We consider the map $\Phi: \mathbf{b} \mapsto f$ where $f$ is a polynomial in our low-rank form (CP decomposition). We bound the Lipschitz constant induced by the $\ell_\infty$ norm on coefficients and the $L^\infty$ norm on functions:
\[
\Lambda(\Phi) = \sup_{\mathbf{b} \neq \mathbf{b}'} \frac{\|f_{\mathbf{b}} - f_{\mathbf{b}'}\|_{L^\infty}}{\|\mathbf{b} - \mathbf{b}'\|_\infty}.
\]

The Bernstein basis offers optimal stability independent of degree $d$. For $x\in[-1,1]$, via the map $s(x) = (x+1)/2\in[0,1]$, the basis is $B_{j,d}(x) = \binom{d}{j} s(x)^j (1-s(x))^{d-j}$. Due to the partition of unity and non-negativity, it satisfies the following.

\begin{lemma}[Univariate Stability]
Let $q(x) = \sum_{j=0}^d b_j B_{j,d}(x)$. The mapping $b \mapsto q$ has a Lipschitz constant of 1:
\[
\|q_{b} - q_{b'}\|_{L^\infty} \le \|b - b'\|_\infty.
\]
\end{lemma}

\begin{proof}
For any $x \in [-1,1]$:
\[
|q_{b}(x) - q_{b'}(x)| = \left|\sum_{j=0}^d (b_j - b'_j) B_{j,d}(x)\right| \le \max_j |b_j - b'_j| \sum_{j=0}^d B_{j,d}(x) = \|b - b'\|_\infty.
\]
\end{proof}

In contrast, the monomial basis is exponentially unstable ($O(2^d)$). Representing bounded functions often requires coefficients of magnitude $2^d$ (the cancellation phenomenon), destroying the conditioning of the moment matrices.

Now, we consider the full objective function in the low-rank sum-of-products format:
\[
f(\mathbf{x}) = \sum_{l=1}^r \underbrace{\prod_{i=1}^n f_{l,i}(x_i)}_{T_l(\mathbf{x})}.
\]
The stability of the mapping from the coefficient vector $\mathbf{b}$ to $f$ depends on the dimension $n$, the rank $r$, and the coefficient bound $M = \|\mathbf{b}\|_\infty$.

\begin{lemma}[LR Coefficient Stability]
If the coefficients of every univariate factor $f_{l,i}$ are bounded by $M$, the Lipschitz constant of the map $\mathbf{b} \mapsto f$ satisfies:
\[
\Lambda_{LR} \le r n  M^{n-1}.
\]
\end{lemma}

\begin{proof}
Let $\tilde{f}_{l,i}$ be factors with perturbed coefficients ($\|\delta\|_\infty \le \epsilon$).
First, for a single rank-1 term $T_l$, using the univariate stability and $\|f_{l,i}\|_\infty \le M$, we bound the product error via a telescoping sum:
\[
\|T_l - \tilde{T}_l\|_\infty \le \sum_{k=1}^n \|f_{l,k} - \tilde{f}_{l,k}\|_\infty \prod_{j \neq k} \|f_{l,j}\|_\infty \le \sum_{k=1}^n \epsilon M^{n-1} = n M^{n-1} \epsilon.
\]
Finally, summing over the rank $r$:
\[
\|f - \tilde{f}\|_\infty \le \sum_{l=1}^r \|T_l - \tilde{T}_l\|_\infty \le \sum_{l=1}^r (n M^{n-1} \epsilon) = r n  M^{n-1}  \epsilon.
\]
\end{proof}

Linear growth with $n$ is the theoretical lower bound for product structures. However, exponential growth in $n$ must be avoided. This is achieved by our initialization strategy.

By choosing Bernstein coefficients with gaps of $O(1/n)$, specifically $1\leq b_j \le 1 + \frac{2\delta}{n}$, we ensure that the coefficient bound $M$ depends on $n$. Substituting $M = 1 + \frac{2\delta}{n}$ into the stability lemma yields:
\[
\Lambda_{LR} \le r  n  \left(1 + \frac{2\delta}{n}\right)^{n-1}.
\]
Using the limit definition of the exponential function, $(1 + \frac{2\delta}{n})^{n-1} \le e^{2\delta}$. Thus:
\[
\Lambda_{LR} \le (r  e^{2\delta})  n.
\]
This certifies that the conditioning of our large-$n$ instances scales linearly with dimension (the term $e^{2\delta}$ is constant with respect to $n$), ensuring numerical solvability but still limited by the number of variables. This is exactly what we see in the results of Table~\ref{tab:bernstein_r2d2} and \ref{tab:bernstein_r3d2}, where the error of the minimum obtained increases with $n$, but not in an exponential way.

\end{document}